\documentclass{amsart}

\usepackage[all]{xypic}
\usepackage{epsf}
\usepackage[centertags]{amsmath}
\usepackage{amsfonts}
\usepackage{amssymb}
\usepackage{amsthm}
\usepackage{newlfont}

 \newtheorem{thm}{Theorem}[section]
 \newtheorem{cor}[thm]{Corollary}
 
 \newtheorem{lem}[thm]{Lemma}
 \newtheorem{prop}[thm]{Proposition}
 \theoremstyle{definition}
 \newtheorem{defn}[thm]{Definition}
 \theoremstyle{remark}
 \newtheorem{rem}[thm]{Remark}
 \theoremstyle{remark}
 
 \numberwithin{equation}{section}

 \newcommand{\St}{\mathrm{St}}
 \newcommand{\re}{\mathrm{Re}}
 
 \newcommand{\id}{\mathrm{id}}
 \newcommand{\an}{\mathrm{an}}

 \newcommand{\Spec}{\mathrm{Spec}}
 \newcommand{\Frob}{\mathrm{Frob}}

 \newcommand{\Res}{\mathrm{Res}}
 \newcommand{\Vol}{\mathrm{Vol}}
 \newcommand{\Gal}{\mathrm{Gal}}
 \newcommand{\GL}{\mathrm{GL}}
 \newcommand{\PGL}{\mathrm{PGL}}

 \newcommand{\pr}{\mathrm{pr}}

 \newcommand{\Tr}{\mathrm{Tr}}
 \newcommand{\et}{\mathrm{et}}

\newcommand{\inv}{\mathrm{inv}}
 
 \renewcommand{\mod}{\mathrm{mod}}

 \newcommand{\fp}{\mathfrak p}

 \newcommand{\fP}{\mathfrak P}
 \newcommand{\cO}{\mathcal{O}}
 
 \renewcommand{\cH}{\mathcal{H}}

 \newcommand{\cE}{\mathcal{E}}
 \newcommand{\cK}{\mathcal{K}}
 
 \newcommand{\cA}{\mathcal{A}}
 
 \newcommand{\cG}{\mathcal{G}}
 
 \newcommand{\cI}{\mathcal{I}}
 
 \newcommand{\cX}{\mathcal{X}}

 \newcommand{\cF}{\mathcal{F}}
 \renewcommand{\cD}{\mathcal{D}}

 \newcommand{\C}{\mathbb{C}}
 \newcommand{\F}{\mathbb{F}}
 \newcommand{\Q}{\mathbb{Q}}
 \newcommand{\Z}{\mathbb{Z}}
 \newcommand{\A}{\mathbb{A}}
 
 \newcommand{\p}{\mathbb{P}}
 \newcommand{\T}{\mathbb{T}}
 \newcommand{\M}{\mathbb{M}}
 \newcommand{\KI}{K_{\cD,I}}
 \newcommand{\bKI}{\bar{K}_{\cD,I}}
 \newcommand{\sKI}{{K}_{\bar{\mathcal{D}},I}}
 \newcommand{\bG}{\bar{G}}
 \newcommand{\Nr}{\mathrm{Nr}}
 \newcommand{\mv}{M^\cD_{I,o}}
 \newcommand{\mvq}{\bar{M}^\cD_{I,o}}
 \newcommand{\chp}{\mathrm{ch.p.}}
 \newcommand{\disc}{\mathrm{disc}}
 
 \newcommand{\wF}{\widetilde{F}}
 \newcommand{\wP}{\widetilde{\Pi}}
 \newcommand{\winf}{\widetilde{\infty}}
 \newcommand{\wo}{\widetilde{o}}
 \newcommand{\wx}{\widetilde{x}}
 
 \newcommand{\Ell}{\mathbf{Ell}}
 
 \newcommand{\G}{\Gamma}
 \newcommand{\To}{\longrightarrow}
 \newcommand{\bs}{\setminus}
 \newcommand{\Fi}{F_\infty}
 
 \newcommand{\bD}{\bar{D}}
 \newcommand{\bcD}{\bar{\mathcal{D}}}

 \newcommand{\norm}[1]{\left\Vert#1\right\Vert}

\begin{document}

\title[Modular varieties
and the Weil-Deligne bound]{Modular varieties of $\cD$-elliptic
sheaves\\ and the Weil-Deligne bound}

\author{Mihran Papikian}

\address{Department of Mathematics, Pennsylvania State University,
University Park, PA 16802}


\email{papikian@math.psu.edu}


\date{}


\begin{abstract} We compare the asymptotic grows of the number of
rational points on modular varieties of $\cD$-elliptic sheaves over
finite fields to the grows of their Betti numbers as the degree of
the level tends to infinity. This is a generalization to higher
dimensions of a well-known result for modular curves. As a
consequence of the main result, we also produce a new asymptotically
optimal sequence of curves.
\end{abstract}

\maketitle

\section{Introduction}

\subsection{Motivation} Let $q$ be a power of a prime $p$ and let $\F_q$ denote
the finite field with $q$ elements. Let $X$ be a smooth, projective,
geometrically irreducible, $d$-dimensional variety over $\F_q$. Fix
an algebraic closure $\overline{\F}_q$ of $\F_q$. Also, fix a prime
number $\ell\neq p$ and an algebraic closure $\overline{\Q}_\ell$ of
the field $\Q_\ell$ of $\ell$-adic numbers. Grothendieck's theory of
\'etale cohomology produces the $\ell$-adic cohomology groups
$$
H^i(X):=H^i(X\otimes_{\F_q} \overline{\F}_q, \overline{\Q}_\ell),
\quad i\geq 0.
$$
These groups are finite dimensional $\overline{\Q}_\ell$-vector
spaces endowed with an action of the Galois group
$\Gal(\overline{\F}_q/\F_q)$. The \textit{$\ell$-adic Betti numbers}
of $X$ are the dimensions
$$
h^i(X):=\dim_{\overline{\Q}_\ell}H^i(X).
$$
It is known that $h^i(X)=0$ for $i>2d$, $h^0(X)=1$, and
$h^{2d-i}(X)=h^i(X)$.

Let $\Frob_q$ be the inverse of the standard topological generator
$x\mapsto x^q$ of $\Gal(\overline{\F}_q/\F_q)$, i.e., the so-called
\textit{geometric Frobenius element}. Assume $H^i(X)\neq 0$. Denote
the eigenvalues of $\Frob_q$ acting on $H^i(X)$ by $\alpha_{i,1},
\alpha_{i,2},\dots, \alpha_{i,s}$ (here $s=h^i(X)$). Deligne proved
that $\{\alpha_{i,j}\}$ are algebraic numbers. Moreover, for any
isomorphism $\iota: \overline{\Q}_\ell\to \C$ the absolute value
$|\iota (\alpha_{i,j})|$ is independent of $\iota$ and is equal to
$q^{i/2}$ (Riemann hypothesis for $X$); see \cite{WeilI}.

For an integer $n\geq 1$ denote by $\F_{q^n}$ the degree $n$
extension of $\F_q$, and let $X(\F_{q^n})$ be the set of
$\F_{q^n}$-rational points on $X$. By the Grothendieck-Lefschetz
trace formula
$$
\# X(\F_{q^n})=\sum_{i=0}^{2d} (-1)^i \Tr(\Frob_q^n\ |\
H^i(X))=\sum_{i=0}^{2d}(-1)^i\sum_{j=1}^{h^i(X)}\alpha_{i,j}^n.
$$
If one combines this with Deligne's result, then there results the
\textit{Weil-Deligne bound}
\begin{equation}\label{eq-WD}
\# X(\F_{q^n}) \leq \sum_{i=0}^{2d}
q^{in/2}h^i(X)=:\mathrm{WD}_n(X).
\end{equation}

It is natural to ask how ``optimal'' is the bound (\ref{eq-WD}).
More precisely, suppose we fix some natural numbers $b_1,
b_2,\cdots, b_d$. \textit{How close can $\# X(\F_{q^n})$ get to
$\mathrm{WD}_n(X)$ for a variety $X$ with $h^i(X)=b_i$, $1\leq i\leq
d$?} Although this question has received a considerable amount of
attention, it still remains largely open, cf. \cite{SerreRP},
\cite{Tsfasman}.

Let $h(X):=\sum_{i=0}^{2d}h^i(X)$. In this paper we will be mostly
concerned with the asymptotic optimality of (\ref{eq-WD}):
\textit{How close can the ratio $\# X(\F_{q^n})/h(X)$ get to
$\mathrm{WD}_n(X)/h(X)$ as $h(X)\to \infty$, assuming $d$, $q$ and
$n$ are fixed?} Tsfasman raised questions of this nature in
\cite{Tsfasman}. Besides its intrinsic mathematical interest, the
motivation for this problem partly comes from coding theory via
Goppa's \cite{Goppa} algebro-geometric construction of
error-correcting codes over $\F_{q^n}$.

\vspace{0.1in}

For curves (i.e., when $d=1$) we clearly have
$\mathrm{WD}_n(X)/h(X)\to q^{n/2}$ as $h(X)\to \infty$. On the other
hand, surprisingly enough, it turns out that
\begin{equation}\label{eq-DV}
\underset{X}{\lim\mathrm{sup}}\left( \frac{\#
X(\F_{q^n})}{h(X)}\right)\leq \frac{q^{n/2}-1}{2}.
\end{equation}
This is a well-known result of Drinfeld and Vladut \cite{DV}. In
particular, curves of genus sufficiently larger than $q^n$ never
have as many rational points as the Weil-Deligne bound allows.

A sequence of curves $\{X_i\}$ over $\F_{q^n}$ is called
\textit{asymptotically optimal} if $h(X_i)\to \infty$ and $\#
X_i(\F_{q^n})/h(X_i)\to (q^{n/2}-1)/2$. It is not known whether
asymptotically optimal sequences exist when $q^n$ is not a square
(even for a singe $q^n$), in other words, it is an open problem
whether in general (\ref{eq-DV}) is the best possible upper-bound on
$\lim\mathrm{sup}(\# X(\F_{q^n})/h(X))$. On the other hand, when
$q^n$ is a square, then asymptotically optimal sequences always
exist. Here Shimura curves (and their function field analogues -
Drinfeld modular curves) play a key role: modular curves with
appropriate level structures over quadratic extensions of finite
fields attain the bound (\ref{eq-DV}) as the size of the level tends
to infinity. This is due to Drinfeld, Ihara, Manin, Tsfasman, Vladut
and Zink; see \cite{Ihara}, \cite{MV}, \cite{TVZ}, \cite{DV}.  In
fact, every known asymptotically optimal sequence of curves over
$\F_{q^2}$ has the property that every $X_i$ is a classical, Shimura
or Drinfeld modular curve for $i$ sufficiently large, cf.
\cite{Elkies}, \cite{many-authors}.

\vspace{0.1in}

For the higher dimensional varieties there are only a few partial
results. Lachaud and Tsfasman, using explicit formulae, proved a
certain analogue of the Drinfeld-Vladut bound (\ref{eq-DV}); see
\cite{LT}, \cite{Tsfasman}. As far as I am aware, there were no
known examples of sequences of $d$-dimensional varieties $\{X_i\}$
such that $h(X_i)\to \infty$ and for which $\# X_i(\F_{q})/h(X_i)$
converges to a number close to the asymptotic Weil-Deligne bound
(assuming $q$ is fixed), besides the following obvious construction.
Take each $X_i$ to be an appropriate product of lower dimensional
varieties. The number of rational points and the Betti numbers are
easy to compute inductively. Indeed, if $X=Y\times Z$ then
$\#X(\F_q)=(\#Y(\F_q))\cdot (\#Z(\F_q))$ and $h(X)=h(Y)\cdot h(Z)$.
(The first identity is clear and the second follows from K\"unneth
formula.) Using this technique, one can produce all sorts of
interesting limits (but they all are smaller than the asymptotic
Weil-Deligne bounds) cf. \cite[$\S$5]{Tsfasman}.

The main result of this paper is a generalization of
Tsfasman-Vladut-Zink result for Shimura curves to the case of
certain higher dimensional modular varieties. The question of
extending the results in \cite{TVZ} to other modular varieties
already appears in that paper (Question C on p.22).

\begin{rem}
One might ask whether the Weil-Deligne bound is ever asymptotically
optimal. More precisely, suppose $d$ and $q$ are fixed. Does there
exist a sequence $\{X_i\}$ of $d$-dimensional varieties over $\F_q$
such that $h(X_i)\to \infty$ and $\#
X_i(\F_{q})/\mathrm{WD}_1(X_i)\to 1$? I expect that the answer is
always negative. When $d=1$ this of course follows from
(\ref{eq-DV}).
\end{rem}

\subsection{Main result}
Let $C:=\p^1_{\F_q}$ be the projective line over $\F_q$. Denote by
$F=\F_q(T)$ the field of rational functions on $C$. Fix some $d\geq
2$ and let $D$ be a $d^2$-dimensional central division algebra over
$F$, which is split at $\infty=1/T$, i.e., $D\otimes_F F_\infty$ is
isomorphic to the algebra $\M_d(F_\infty)$ of $d\times d$ matrices
with entries in $F_\infty$. Fix a locally-free sheaf $\cD$ of
$\cO_C$-algebras on $C$ whose generic fibre is $D$ and such that for
every place $x\in |C|$, $\cD_x:=\cD\otimes_{\cO_C} \cO_{x}$ is a
maximal order of $D_x=D\otimes_F F_x$, where $\cO_x$ and $F_x$ are
the completions of $\cO_{C,x}$ and $F$ at $x$, respectively. Denote
by $R\subset |C|$ the set of places where $D$ ramifies; hence for
all $x\not\in R$ the couple $(D_x,\cD_x)$ is isomorphic to
$(\M_d(F_x),\M_d(\cO_x))$. \textit{In this paper we make a blanket
assumption that $D_x$ is a division algebra for every $x\in R$}. Let
$o$ be a fixed closed point on $C-R-\{\infty\}$. Denote the residue
field at $o$ by $\F_o$, its degree $m$ extension by $\F_o^{(m)}$,
and $q_o:=q^{[\F_o:\F_q]}=\# \F_o$.

Call a prime ideal $\fp$ of the polynomial ring $A:=\F_q[T]$
\textit{admissible} if $x\mapsto x^d$ is an automorphism of
$(A/\fp)^\times/\F_q^\times$. There are infinitely many admissible
primes; see \cite[Lem. 4.6]{PapIMRN}. A closed subscheme $I$ of $C$
is called \textit{admissible} if $I\cap (R\cup \infty\cup
o)=\emptyset$ and $I=\Spec(A/\fp)$ for an admissible prime $\fp\lhd
A$.

Assume $I=\Spec(A/\fp)$ is admissible and let $\mv$ be the modular
variety of isomorphism classes of $\cD$-elliptic sheaves over
$\overline{\F}_o$ with pole at $\infty$ equipped with level-$I$
structures; see $\S$\ref{sec2} for definitions. There is a natural
action of $(A/\fp)^\times$ on $\mv$ via its (diagonal) action on
level structures. We denote the quotient variety by $\mvq$. The main
result of this paper is the following:

\begin{thm}\label{thm-main} $\mvq$ is a smooth, projective, geometrically irreducible,
$(d-1)$-dimensional variety defined over $\F_o$. Moreover
\begin{equation}\label{eq-my}
\lim_{\deg(I)\to
\infty}\frac{\#\mvq(\F_o^{(d)})}{h(\mvq)}=\frac{1}{d}\prod_{i=1}^{d-1}
(q_o^i-1).
\end{equation}
\end{thm}

A similar result for Drinfeld modular varieties is proven in
\cite{PapIMRN}.

It will be clear from the proof of Theorem \ref{thm-main} that
$h(\mvq)\to \infty$ as $\deg(I)\to \infty$. If we specialize the
theorem to $d=2$ and $\deg(o)=1$, then we get a sequence of curves
over $\F_{q^2}$, indexed by the primes in $A$ of odd degree, which
is asymptotically optimal. This last fact is new and is of
independent interest due to a rather limited number of examples of
asymptotically optimal sequences of curves.

From one of the main results in \cite{LRS} (see Theorem
\ref{thm6.1}) or, alternatively, from the discussion in
$\S$\ref{ss6.2}, it follows that
\begin{equation}\label{LRS-limit}
\lim_{\deg(I)\to
\infty}\frac{\mathrm{WD}_d(\mvq)}{h(\mvq)}=q_o^{d(d-1)/2}.
\end{equation}

One can compare the limits (\ref{eq-my}) and (\ref{LRS-limit}) from
two opposite viewpoints. On the one hand, since
$\frac{1}{d}\prod_{i=1}^{d-1}(q_o^{i}-1) < q_o^{d(d-1)/2}$, modular
varieties $\mvq$ never have as many $\F_o^{(d)}$-rational points as
the Weil-Deligne bound allows when $\deg(I)$ is large. On the other
hand, the degree of $\frac{1}{d}\prod_{i=1}^{d-1}(q_o^{i}-1)$ as a
polynomial in $q_o$ is the same as the degree of $q_o^{d(d-1)/2}$,
so $\#\mvq(\F_o^{(d)})$ asymptotically comes close to
$\mathrm{WD}_d(\mvq)$, especially when $q_o$ is large compared to
$d$, and one can say that the varieties $\mvq$ have many
$\F_o^{(d)}$-rational points compared to their Betti numbers.

\subsection{Shimura varieties}
Modular varieties of $\cD$-elliptic sheaves are the function field
analogues of Shimura varieties over number fields. The anonymous
referee pointed out that in a recent paper \cite{Sauvageot} Fran\c
cois Sauvageot considered a problem about asymptotic properties of
Shimura varieties which is similar to the problem we address in this
paper. In this subsection we would like to compare Sauvageot's
result to ours.

Let $G$ be an algebraic connected reductive group over a global
field of characteristic $0$. Sauvageot introduces a notion of
\textit{strongly vanishing family} $\cK$ of compact subgroups of $G$
over the finite adeles. Attached to $G$ and $K\in \cK$ there is a
Shimura variety $X_K$. Generalizing an approach of Serre
\cite{SerreJAMS} by means of Arthur's trace formula, Sauvageot gives
a simple expression for $\underset{K\in \cK}{\lim}\Vol(K)\chi(X_K)$
in terms of an invariant of $G$ and $\cK$; here $\chi(X_K)$ is the
Euler-Poincar\'e characteristic of the $L^2$-cohomology of $X_K$ and
$\Vol(K)$ is the volume of $K$ with respect to an appropriately
normalized Haar measure. On the other hand, for a prime of good
reduction $o$ of $X_K$, assuming a conjecture of Milne, a formula of
Kottwitz expresses $\underset{K\in \cK}{\lim}\Vol(K)(\#
X_{K,o}(\F_o^{(m)}))$, $m\geq 1$, in terms of an invariant of $G$
and a sum of twisted orbital integrals. This again uses trace
formula techniques. Overall, one obtains a (conjectural) formula for
$\underset{K\in \cK}{\lim}\# X_{K,o}(\F_o^{(m)})/\chi(X_K)$ which
involves a sum of twisted orbital integrals. Then Sauvageot gives
some criteria for detecting strongly vanishing families $\cK$ (e.g.,
for $\GL_2$ over $\Q$ the classical congruence subgroups form such
families as the level tends to infinity). Some relevant questions
are not discussed in \cite{Sauvageot}: First, whether it is possible
to compute $\underset{K\in \cK}{\lim}\#
X_{K,o}(\F_o^{(m)})/\chi(X_K)$ explicitly in some situations other
than the case of curves, i.e., to give a simple expression for the
sum of twisted orbital integrals which comes from Kottwitz's
formula. Second, what is the relationship of $\chi(X_K)$ to the
invariants of $X_{K,o}$, such as the $\ell$-adic Euler-Poincar\'e
characteristic of $X_{K,o}$, especially when $X_K$ is not compact.
Third, for which $K$ the resulting varieties $X_{K,o}$ are smooth
and geometrically irreducible.

Now from the point of view of general Shimura varieties the
situation which we consider in this paper is rather special since
the only $G$ which is allowed is the multiplicative group of a
division algebra. Our notion of admissible level seems to be
analogous to Sauvageot's notion of vanishing family. The main
advantage of our result is that the asymptotic formula which we
obtain is very explicit. Our strategy of the proof is also
different. In computing the asymptotic number of rational points
over finite fields we crucially rely on the modular interpretation
of the points on $M^\cD_{I,o}$. In computing the asymptotic Betti
numbers we use special properties of the cohomology of varieties
having non-archimedean uniformization, along with some classical
results about discrete subgroups of $p$-adic groups. Note that
estimating the sum of Betti numbers $h(\mvq)$ or their alternating
sum $\chi(\mvq)$ are equivalent problems for $\mvq$ since the middle
cohomology group grows much faster than all the others (although in
general $h$ is certainly a better measure of the combined size of
the Betti numbers than $\chi$). It is reasonable to expect that the
methods of this paper can be adapted to some PEL Shimura varieties.

\subsection{Outline of the proof} The proof of Theorem
\ref{thm-main} consists of three parts, which are more-or-less
independent of each other.

In Section \ref{sec2} we recall the definition of moduli schemes of
$\cD$-elliptic sheaves and discuss the Stein factorization of these
schemes. We prove that $\mvq$ is smooth and geometrically
irreducible when $I$ is admissible.

In Section \ref{sec3} we discuss $\cD$-elliptic sheaves over
$\overline{\F}_o$ and the rationality of the corresponding points on
the moduli varieties. The main tool used in this section is the
description of the set of isomorphism classes of $\cD$-elliptic
sheaves in a given isogeny class \cite[\S\S9-10]{LRS}. We consider a
special class of $\cD$-elliptic sheaves, called supersingular
$\cD$-elliptic sheaves, and show that the corresponding points on
$\mvq$ are all $\F_o^{(d)}$-rational. Moreover, we prove that when
$\deg(I)$ is large enough, compared to $d$ and $\deg(o)$, the
supersingular points are the only $\F_o^{(d)}$-rational points on
$\mvq$ (when $d=2$ this can be replaced by the Drinfeld-Vladut
bound). We should mention that Lachaud-Tsfasman bound is essentially
equivalent to the Weil-Deligne bound when applied to $\mvq$, hence
is much larger than the limit (\ref{eq-my}). The number of
supersingular points can be computed as a volume. Overall, we get an
asymptotic formula for $\#\mvq(\F_o^{(d)})$.

In Section \ref{sec4} we estimate how $h(\mvq)$ grows as $\deg(I)$
tends to infinity. This relies on several deep results. First, using
proper base change, we transfer the problem from $o$-fibre to
$\infty$-fibre. Then a result of Schneider and Stuhler about
cohomology of varieties having rigid-analytic uniformization (in
combination with Berkovich's theory) reduces the problem to a
calculation of the dimension of a certain space of cusp forms on
$D^\times$. Next, a theorem of Casselman and Garland reduces the
calculation of this dimension to a calculation of the
Euler-Poincar\'e characteristic of a certain discrete cocompact
subgroup of $\PGL_d(\Fi)$. Finally, Serre's theory of
Euler-Poincar\'e measures allows us to compute the Euler-Poincar\'e
characteristic as a volume.



\section{Notation}

Aside from the notation in Introduction, we will use the following:

\subsection{} The residue field of $\cO_x$ is
denoted by $\F_x$, the cardinality of $\F_x$ is denoted by $q_x$. We
assume that the valuation $v_x:F_x\to\Z$ is normalized by
$v_x(\varpi_x)=1$, where $\varpi_x$ is a uniformizer of $\cO_x$; the
norm $|\cdot|_x$ on $F_x$ is $q_x^{-v_x(\cdot)}$. We denote the
adele ring of $F$ by $\A:=\prod'_{x\in |C|} F_x$. For a set of
places $S$ of $F$ we denote by $\A^S:=\prod'_{x\not\in S}F_x$ the
adele ring outside $S$, and $\cD^S:=\prod_{x\not\in S}\cD_x$.

\subsection{}\label{ss2.2} Let $I\neq \emptyset$ be a closed subscheme of $C$, and
let $\cI$ be the ideal sheaf of $I$. Denote $\cO_I=\cO_C/\cI$,
$\deg(I)=\dim_{\F_q}(\cO_I)$, and $\cD_I=\cD\otimes_{\cO_C}\cO_I$.
For $I$ is disjoint from $S$, let
$$
K_{\cD,I}^S:=\ker((\cD^S)^\times\to \cD_I^\times).
$$

\subsection{} Let $\zeta_F(s)=\prod_{x\in |C|}\zeta_x(s)$ be the zeta
function of $C$; here $\zeta_x(s)=(1-q_x^{-s})^{-1}$. It is
well-known (and is easy to prove) that
$$
\zeta_F(s)=\frac{1}{(1-q^{-s})(1-q^{1-s})},
$$
Let $\zeta_F^S(s)=\prod_{x\not\in S}\zeta_x(s)$ be the partial zeta
function with respect to $S$.

\subsection{} For a scheme $W$ over $\F_q$ denote by $\Frob_W$ its
Frobenius endomorphism, which is the identity on the points and the
$q$-th power map on the functions. Denote by $C\times W$ the fibred
product $C\times_{\Spec(\F_q)}W$. For a sheaf $\cF$ on $C$ and a
sheaf $\cG$ on $W$, the sheaf $\pr_1^\ast(\cF)\otimes
\pr_2^\ast(\cG)$ is denoted by $\cF\boxtimes \cG$.

\subsection{} Let $G$ be the
algebraic group over $F$ defined by $G(B)=(D\otimes_F B)^\times$ for
any $F$-algebra $B$; this is the multiplicative group of $D$.


\section{Moduli schemes of $\cD$-elliptic sheaves}\label{sec2}

Let $S$ be a $\F_q$-scheme. A $\cD$-\textit{elliptic sheaf} over $S$
consists of a commutative diagram
\begin{equation}\label{eq-es}
\xymatrix{ \cdots \ar@{^{(}->}[r]^-{j_{i-2}} & \cE_{i-1}
\ar@{^{(}->}[r]^-{j_{i-1}} &
\cE_{i} \ar@{^{(}->}[r]^-{j_i} & \cE_{i+1} \ar@{^{(}->}[r]^-{j_{i+1}} & \cdots\\
\cdots \ar[ur]^-{t_{i-2}}\ar@{^{(}->}[r]^-{^\tau\!j_{i-2}} &
{^\tau}\!\cE_{i-1}
\ar[ur]^-{t_{i-1}}\ar@{^{(}->}[r]^-{^\tau\!j_{i-1}} &
{^\tau}\!\cE_{i} \ar[ur]^-{t_i}\ar@{^{(}->}[r]^-{^\tau\!j_i} &
{^\tau}\!\cE_{i+1}
\ar[ur]^-{t_{i+1}}\ar@{^{(}->}[r]^-{^\tau\!j_{i+1}} & \cdots}
\end{equation}
where each $\cE_i$ is a locally free $\cO_{C\times S}$-module of
rank $d^2$ equipped with a right action of $\cD$ compatible with the
$\cO_C$-action,
$$
^\tau\!\cE_i=(\id_C\times \Frob_{S})^\ast \cE_i,
$$
and all $j$'s and $t$'s are $\cO_{C\times S}$-linear injections
compatible with the action of $\cD$. All that is given is subject to
the following conditions:
\begin{enumerate}
\item[(i)] Periodicity:
$$
\cE_{i+d}=\cE_i\otimes_{\cO_{C\times S}} (\cO_C(\infty)\boxtimes
\cO_S).
$$
Here $\cE_i$ is considered as a submodule of $\cE_{i+d}$ under the
$d$-fold composition of $j$, and $\cO_C(\infty)$ is an
$\cO_C$-module via the natural injection $\cO_{C}\hookrightarrow
\cO_{C}(\infty)$.
\item[(ii)] Pole: $\cE_i/j_{i-1}(\cE_{i-1})$ is isomorphic to the direct
image $(\pi_\infty)_\ast\cG_i$ of a locally free rank-$d$
$\cO_S$-module $\cG_i$ by the $\infty$ section:
$$
\pi_\infty: S\to C\times S, \quad s\mapsto (\infty, s).
$$
\item[(iii)] Zero: $\cE_i/t_{i-1}({^\tau\!\cE_{i-1}})$ is isomorphic to the direct image
$(\pi_z)_\ast\cH_i$ of a locally free rank-$d$ $\cO_S$-module
$\cH_i$ by the section
$$
\pi_z:S\to C\times S,\quad s\mapsto (z(s),s),
$$
where $z:S\to C$ is a morphism of $\F_q$-schemes such that
$z(S)\subset C-R-\{\infty\}$.
\end{enumerate}

Let $I\neq \emptyset$ be a closed subscheme of $C-R-\{\infty\}$. Let
$(\cE_i,j_i,t_i)$ be a $\cD$-elliptic sheaf over $S$ such that
$z(S)$ is disjoint from $I$. The restriction $\cE_I:=\cE_{i|I\times
S}$ is independent of $i$, and $t$ induces an isomorphism
$^\tau\!\cE_I\cong \cE_I$. A \textit{level-$I$ structure} on
$(\cE_i,j_i,t_i)$ is an $\cO_{I\times S}$-linear isomorphism $\iota:
\cD_I\boxtimes \cO_S\cong \cE_I$, compatible with the action of
$\cD_I$, which makes the following diagram commutative:
$$
\xymatrix{ ^\tau\!\cE_I\ar[rr]^-t & &\cE_I\\
& \cD_I\boxtimes \cO_S \ar[lu]^-{^\tau\!\iota}\ar[ru]_-\iota&}
$$

For a scheme $$z:S\to C':=C-I-R-\{\infty\},$$ denote by
$\Ell_{I}^\cD(S)$ the set of isomorphism classes of $\cD$-elliptic
sheaves over $S$ with level-$I$ structures. There are natural
commuting actions of $\Z$ and $\cD_I^\times$ on $\Ell_{I}^\cD(S)$:
$n\in \Z$ acts by
$$
[n](\cE_i,j_i,t_i;\iota)=(\cE_{i+n},j_{i+n},t_{i+n};\iota)
$$
and $g\in \cD_I^\times$ acts by
$$
(\cE_i,j_i,t_i;\iota)g=(\cE_{i},j_{i},t_{i};\iota\circ g),
$$
where $g$ acts on $\cD_I\boxtimes \cO_S$ via right multiplication on
$\cD_I$.

By (4.1), (5.1) and (6.2) in \cite{LRS}, the functor $S\mapsto
\Ell^{\cD}_{I}(S)/\Z$ is representable by a smooth, projective
scheme $M^\cD_I$ over $C'$ of pure relative dimension $(d-1)$. The
action of $\cD_I^\times$ on $\Ell^\cD_{I}$ induces an action of this
finite group on $M^\cD_I$. We denote by $\bar{M}_I^\cD$ the quotient
of $M_I^\cD$ under the action of the center $Z(\cD_I^\times)\cong
\cO_I^\times$ of $\cD_I^\times$.

Note that if we put $d=1$ and apply the previous definitions with
$\cD=\cO_C$, then we arrive at the notion of $\cO_C$-elliptic
sheaves with level-$I$ structures. This case was considered much
earlier by Drinfeld \cite{Drinf.ES}, who proved that $M_I^{\cO_C}$
is isomorphic to the moduli scheme of rank-$1$ Drinfeld $A$-modules
with level-$I$ structures. This latter moduli scheme is closely
related to Class Field Theory of $F$; see \cite[Thm. 1]{Drinfeld}.

As follows from \cite[pp. 26-29]{Lafforgue}, there is a natural
morphism of schemes over $C'$
$$
\wp: M^\cD_I\to M^{\cO_C}_I
$$
which is compatible with the action of $\cD_I^\times$ in the sense
that $\wp\circ g=\det(g)\circ \wp$ for any $g\in\cD_I^\times\cong
\GL_d(\cO_I)$.

\begin{prop}\label{cor1.6}
The fibres of $\wp$ are geometrically irreducible.
\end{prop}
\begin{proof}
By Stein factorization theorem, it is enough to show that the fibres
of
$$
\wp_{\bar{\eta}}:
M^\cD_{I,\bar{\eta}}:=M^\cD_I\times_{C'}\Spec(\bar{F})\to
M^{\cO_C}_{I,\bar{\eta}}:=M^{\cO_C}_I\times_{C'}\Spec(\bar{F})
$$
are connected, where $\bar{F}$ denotes a fixed algebraic closure of
$F$. Since $\wp_{\bar{\eta}}$ is $\cD_I^\times$-equivariant and
$\det:\cD_I^\times\to \cO_I^\times$ is surjective,
$\wp_{\bar{\eta}}$ is surjective. Hence it is enough to show that
the number of connected components of $M^\cD_{I,\bar{\eta}}$ and
$M^{\cO_C}_{I,\bar{\eta}}$ are the same.

By Class Field Theory, $M^{\cO_C}_{I,\bar{\eta}}$ is a disjoint
union of $\#\left[F^\times\bs
(\A^\infty)^\times/K^\infty_{\cO_C,I}\right]$ copies of
$\Spec(\bar{F})$ and $\cO_I^\times$ acts transitively on these
points. On the other hand, the number of connected components of
$M^\cD_{I,\bar{\eta}}$ is equal to the 0-th Betti number
$h^0_{I,\eta}$ (in the notation of Section \ref{sec4}), so the
desired claim follows from Corollary \ref{cor-new}.
\end{proof}

\begin{cor}\label{cor-SF}
If $I=\Spec(A/\fp)$ is admissible then $\mvq$ is a smooth,
projective, geometrically irreducible, $(d-1)$-dimensional variety
defined over $\F_o$, which is a form of one of the components of
$\mv$.
\end{cor}
\begin{proof}
The group $\F_\fp^\times/\F_q^\times$ acts freely and transitively
on the geometrically irreducible components of $M^{\cO_C}_{I,o}$. On
the other hand, it is easy to see that the action of
$Z(\cD_I^\times)\cong \F_\fp^\times$ on $\mv$ factors through
$\F_\fp^\times/\F_q^\times$. By Proposition \ref{cor1.6},
$\wp_o\circ z=z^d\circ \wp_o$ for $z\in \F_\fp^\times$. Since $\fp$
is admissible, $(\F_\fp^\times)^d$ surjects onto
$\F_\fp^\times/\F_q^\times$. We conclude that
$\F_\fp^\times/\F_q^\times$ acts freely and transitively also on the
geometrically irreducible components of $\mv$. This implies the
claim.
\end{proof}


\section{Volume calculation} This section is of auxiliary nature. Here we compute a certain
volume which is used in Sections \ref{sec3} and \ref{sec4}. This
result should be well-known, but in absence of a convenient explicit
reference we sketch some of the details.

For $x\in |C|$, normalize the Haar measure $dg_x$ on $G(F_x)$ by
$\Vol(\cD_x^\times, dg_x)=1$. Fix the Haar measure $d\bar{g}$ on
$G(\A)$ to be the restricted product measure. This measure will be
called the \textit{canonical product measure}.

Consider the homomorphism
\begin{equation}\label{eq-norm}
\norm{\cdot}:G(\A)\to q^{\Z}
\end{equation}
given by the composition of the reduced norm $\Nr: G(\A)\to
\A^\times$ with the idelic norm $\prod_{x\in |C|}|\cdot|_x:
\A^\times \to q^\Z$. Denote the kernel of this homomorphism by
$G^{1}(\A)$. The group $G(F)$, under the diagonal embedding into
$G(\A)$, lies in $G^{1}(\A)$, thanks to the product formula. The
quotient $G(F)\bs G^1(\A)$ is compact, hence has finite volume. The
main result of this section is the following:

\begin{prop}\label{prop4.1}
$$
\Vol\left(G(F)\bs G^1(\A),
d\bar{g}\right)=\frac{1}{(q-1)}\prod_{i=1}^{d-1}\zeta^R_F(-i).
$$
\end{prop}
\begin{proof}
Let $\Phi$ be the characteristic function of $\cD$ in $D(\A)$.
Consider the following integral
$$
\zeta_{D}(s)=\int_{G(\A)} \Phi(g)\norm{g}^s d\bar{g}.
$$
It absolutely converges for $\re(s)>1$, can be meromorphically
continued to the whole plane with a simple pole at $0$, cf.
\cite[$\S$3.1]{WeilAdeles}. Moreover, from the calculations in
\textit{loc.cit.}, one can deduce that
\begin{equation}\label{eq1}
\underset{s=0}{\Res}\ \zeta_{D}(s)=-\Vol\left(G(F)\bs G^1(\A),
d\bar{g}\right) \frac{1}{\log q}.
\end{equation}
Next, we have the decompositions $\Phi=\prod_x \Phi_x$ and
$\zeta_{D}(s)=\prod_x \zeta_{D_x}(s)$, where $\Phi_x$ is the
characteristic function of $\cD_x$ in $D_x$ and
$$
\zeta_{D_x}(s)=\int_{G(F_x)} \Phi_x(g_x)\cdot |\Nr(g_x)|_x^s dg_x.
$$
If $D$ is split at $x$ then, considering the decomposition of
$\M_d(\cO_x)$ into left $\cD_x^\times\cong\GL_d(\cO_x)$-cosets, we
get
\begin{align*}
\zeta_{D_x}(s)
&=\int_{\M_d(\cO_x)-\{0\}}|\Nr(g_x)|_x^s d g_x\\
&=\Vol(\cD_x^\times)\sum_{(n_i)\in \Z^d_{\geq
0}}q_x^{(-sn_1+(-s+1)n_2+\cdots+(-s+d-1)n_d)}\\&=\zeta_x(s)\cdot\zeta_x(s-1)\cdots\zeta_x(s-(d-1)).
\end{align*}
On the other hand, if $D$ is ramified at $x$ then $D_x$ is a
division algebra by assumption, hence $\cD_x$ has a unique maximal
ideal $\fP\lhd \cD_x$ and $\Nr(\fP)=\varpi_x$; see \cite[Thm.
24.13]{Reiner}. Thus,
$$
\zeta_{D_x}(s) = \int_{\cD_x-\{0\}}|\Nr(g_x)|_x^s
dg_x=\Vol(\cD_x^\times)\sum_{n\in \Z_{\geq 0}} q_x^{-sn}=
\zeta_x(s).
$$
Combining these local calculations,
\begin{equation}\label{eq2}
\underset{s=0}{\Res}\ \zeta_{D}(s)=-\frac{1}{(q-1)\log
q}\prod_{i=1}^{d-1}\zeta_F^R(-i).
\end{equation}
From (\ref{eq1}) and (\ref{eq2}) we deduce
$$
\Vol\left(G(F)\bs G^1(\A),
d\bar{g}\right)=\frac{1}{(q-1)}\prod_{i=1}^{d-1}\zeta_{F}^R(-i),
$$
as was required.
\end{proof}


\section{Supersingular $\cD$-elliptic sheaves}\label{sec3}

In this section we discuss $\cD$-elliptic sheaves over
$k:=\overline{\F}_o$ and the rationality of the corresponding points
on the moduli schemes.


\subsection{Isogeny classes} One of the key preliminary results in
\cite{LRS} is the description of the points on the closed fibres of
$M^\cD_{I}\to C'$. This is done in two steps, similar to the
description of the set of abelian varieties over finite fields: one
starts by describing the isogeny classes of $\cD$-elliptic sheaves
over $k$ (as in Honda-Tate theory) and then parametrizes
$\cD$-elliptic sheaves in each isogeny class. We start by recalling
this description.

\begin{defn}(\cite[(9.11)]{LRS}) A \textit{$(D,\infty,o)$-type} is a pair $(\wF,\wP)$, where
$\wF$ is a finite separable field extension of $F$ and $\wP\in
\wF^\times\otimes_\Z \Q$, satisfying the following conditions:
\begin{itemize}
\item For a proper subfield $\wF'\subsetneqq \wF$, $\wP\not\in
(\wF')^\times\otimes_\Z \Q$.
\item $[\wF:F]$ divides $d$.
\item $F_\infty\otimes_F \wF$ is a field and, if $\winf$ is the
unique place of $\wF$ which divides $\infty$, we have
$$
\deg(\winf)\cdot v_{\winf}(\wP)=-[\wF:F]/d.
$$
\item There exists a unique place $\wo\neq \winf$ of $\wF$ such that
$v_{\wo}(\wP)\neq 0$; moreover $\wo$ divides $o$.
\item For each place $x$ of $F$ and each $\widetilde{x}$ of $\wF$
dividing $x$, we have
$$
(d[\wF_{\widetilde{x}}:F_x]/[\wF:F])\cdot \inv_x(D)\in \Z.
$$
\end{itemize}
\end{defn}

In \cite[(9.2)]{LRS} the authors introduce the notion of isogenies
between $\cD$-elliptic sheaves over $k$. We will not recall this
definition; for our purposes it is enough to know \cite[(9.13)]{LRS}
that there is a canonical bijection between the set of isogeny
classes of $\cD$-elliptic sheaves over $k$ and the set of
isomorphism classes of $(D,\infty,o)$-types.

Assume $I\subset C-R-\{o,\infty\}$ and denote by $
\mv:=M^\cD_{I}\times_{C'}\Spec(\F_o)$ the fibre of $M^\cD_{I}$ over
$o$. Fix a $(D,\infty,o)$-type $(\wF,\wP)$ and denote by
$$
\mv(k)_{(\wF,\wP)}\subset \mv(k)
$$
the set of isomorphism classes of $\cD$-elliptic sheaves over $k$
which are in the isogeny class corresponding to $(\wF,\wP)$.

Let $\Delta$ be the central division algebra over $\wF$ with
invariants
$$
\inv_{\widetilde{x}}\Delta=\left\{
               \begin{array}{ll}
                 [\wF:F]/d, & \hbox{if $\wx=\winf$;} \\
                 -[\wF:F]/d, & \hbox{if $\wx=\wo$;} \\
                 {[\wF_{\wx}:F_x]}\cdot\inv_x(D), & \hbox{otherwise.}
               \end{array}
             \right.
$$
Let
$$
h=[\wF_{\wo}:F_o]d/[\wF:F].
$$
$\Delta^\times$ naturally acts on the Dieudonn\'e modules of a
$\cD$-elliptic sheaf in the isogeny class of $(\wF, \wP)$, and these
actions induce group homomorphisms
$$
\left\{
               \begin{array}{ll}
                 \Delta^\times\hookrightarrow G(\A^{\infty,o})\\
                 \Delta^\times\hookrightarrow \GL_{d-h}(F_o)\\
                 \Delta^\times\hookrightarrow
N_{o,h}^\times \overset{\Nr}{\To}F_o^\times\overset{v_o}{\To}\Z,
               \end{array}
             \right.
$$
where $N_{o,h}$ is the central division algebra over $F_o$ with
invariant $-1/h$; see \cite[p. 270]{LRS}. The main result of
\cite[$\S$10]{LRS} is the following:
\begin{thm}\label{thm-combin} There is a bijection
$$
\mv(k)_{(\wF,\wP)}\overset{\sim}{\To}\Delta^\times\bs(Y_I^{\infty,o}\times
Y_o^{\wo}\times Y_{\wo}),
$$
compatible with the action of $\cD_I^\times$, where
\begin{align*}
Y_I^{\infty,o}:=G(\A^{\infty,o})/\KI^{\infty,o},\quad
Y_o^{\wo}:=\GL_{d-h}(F_o)/\GL_{d-h}(\cO_o),\quad Y_{\wo}:=\Z.
\end{align*}
The action of $\Frob_o$ on $\mv(k)_{(\wF,\wP)}$ corresponds to
translation by $1$ on $Y_{\wo}=\Z$.
\end{thm}

\subsection{Rational points}
We say that a $\cD$-elliptic sheaf over $k$ is
\textit{supersingular} if in its $(D,\infty,o)$-type
$\widetilde{F}=F$. It is not hard to show that all supersingular
$\cD$-elliptic sheaves are isogenous \cite[Prop. 10.2.1]{Boyer},
i.e., there is a unique $(D,\infty,o)$-type $(F,\Pi)$.

We have an action of $\cD_I^\times\cong \GL_d(\cO_I)$ on $\mv$ via
its action on level structures in the moduli problem. Denote the
quotient of $\mv$ under the action of the center
$Z(\cD_I^\times)\cong \cO_I^\times$ of $\cD_I^\times$ by
$$
\mvq:=\mv/Z(\cD_I^\times).
$$
Denote the preimage of $Z(\cD_I^\times)$ in
$(\cD^{\infty,o})^\times$ under
$(\cD^{\infty,o})^\times\twoheadrightarrow \cD_I^\times$ by
$\bKI^{\infty,o}$, and denote by $\mvq(k)_{(\wF,\wP)}$ the image of
$\mv(k)_{(\wF,\wP)}$ under the quotient map $M^\cD_{I,o}\to
\bar{M}^\cD_{I,o}$.

\begin{prop}\label{prop-q.c}
There is a bijection
$$
\mvq(k)_{(\wF,\wP)}\overset{\sim}{\To}\Delta^\times\bs(\bar{Y}_I^{\infty,o}\times
Y_o^{\wo}\times Y_{\wo}),
$$
where $\bar{Y}_I^{\infty,o}:=G(\A^{\infty,o})/\bKI^{\infty,o}$. The
action of $\Frob_o$ on $\mvq(k)_{(\wF,\wP)}$ corresponds to
translation by $1$ on $Y_{\wo}=\Z$.
\end{prop}
\begin{proof}
This is a consequence of Theorem \ref{thm-combin}.
\end{proof}

From this proposition it is not hard to deduce the following
criterion for the existence of $\F_o^{(n)}$-rational points on
$\mvq$, cf. \cite[p. 278]{LRS}:
$$
P=\Delta^\times[g^{\infty,o}\bar{K}^{\infty,o}_{\cD,I},g^{\wo}_o\GL_{d-h}(\cO_o),m_{\wo}]\in
\mvq(k)_{(\wF,\wP)}.
$$
is rational over $\F_o^{(n)}$ if and only if there exists $\delta\in
\Delta^\times$ such that
\begin{equation}\label{conds}
\left\{
               \begin{array}{ll}
                 (g^{\infty,o})^{-1}\delta g^{\infty,o}\in \bKI^{\infty,o}\\
                 (g^{\wo}_o)^{-1}\delta g^{\wo}_o\in \GL_{d-h}(\cO_o)\\
                 v_{\wo}(\Nr(\delta))=n\deg(o)/\deg(\wo)
               \end{array}
             \right.
\end{equation}
where $\Nr$ is the reduced norm on $\Delta$.

\vspace{0.1in}

Let $B$ be a finite dimensional $K$-algebra, where $K$ is a field.
The left multiplication by $\alpha\in B$ gives a $K$-linear
transformation of $B$ (as a finite dimensional vector space over
$K$). Define the characteristic polynomial $\chp_{B/K}\alpha\in
K[X]$ of $\alpha$ to be the characteristic polynomial of this
transformation, and $\det_{B/K}(\alpha)$ be its determinant.

\begin{prop}\label{prop2.3}
Suppose $n$ is fixed and $I$ has in its support a prime $\fp\lhd A$
whose degree is large enough compared to $n$. Then
$$
\mvq(\F_o^{(n)})=\mvq(\F_o^{(n)})_{(F,\Pi)}.
$$
In other terms, the images of supersingular points on $\mvq$ are the
only possible $\F_o^{(n)}$-rational points.
\end{prop}
\begin{proof} Suppose $P\in
\mvq(\F_o^{(n)})_{(\wF,\wP)}$. We need to show $\wF=F$. Let
$\delta\in\Delta^\times$ be an element corresponding to $P$. Let
$F'=F[\delta]$ be the field generated by $\delta$ over $F$ and
$f_\delta\in F[X]$ be the minimal polynomial of $\delta$. The
conditions in (\ref{conds}) imply that $\delta$ is integral over
$A$, so $f_\delta$ is a monic irreducible polynomial in $A[X]$.

First, assume $F'/F$ is separable. Let $\disc(\delta)\in A$ be the
discriminant of $f_\delta$. We can consider $\Delta$ as a finite
dimensional algebra over $F$. Since $f_\delta$ divides
$\chp_{\Delta/F}\delta$, the constant term $c_0(\delta)$ of
$f_\delta$ divides $\det_{\Delta/F}(\delta)$. By
\cite[(9.12)]{Reiner},
$$
\mathrm{det}_{\Delta/F}(\delta)=\Nr_{\wF/F}(\mathrm{det}_{\Delta/\wF}(\delta)).
$$
On the other hand, since $\Delta$ is a central division algebra over
$\wF$ of dimension $\left(\frac{d}{[\wF:F]}\right)^2$,
$$
\mathrm{det}_{\Delta/\wF}(\delta)=\Nr(\delta)^{d/[\wF:F]}.
$$
Since $\Nr(\delta)$ has zero valuation at every finite place of
$\wF$ except at $\wo$ where the valuation is bounded in terms of
$n$, $d$ and $\deg(o)$, we conclude that $\deg(c_0(\delta))$ is
bounded in terms of $n$, $d$ and $\deg(o)$. The degree of $f_\delta$
is bounded by $d$, so from the definition of discriminant,
$\deg(\disc(\delta))\leq d(d-1)\deg(c_0(\delta))$. We conclude that
$\deg(\disc(\delta))$ is bounded in terms of $n$, $d$ and $\deg(o)$.

On the other hand, the fact that $\delta$ can be conjugated to lie
in $\bKI^{\infty,o}$ implies that $f_\delta$ modulo $\fp$ is a power
of a linear polynomial. Therefore, $\fp$ must divide
$\disc(\delta)$, unless $f_\delta$ is a linear polynomial. Since the
degree of the discriminant of $\delta$ is bounded, if $\deg(\fp)$ is
large enough compared to $n$, then $f_\delta$ has to be a linear
polynomial, so $\delta\in F$ and $F'=F$. Finally, by \cite[Lem.
3.3.6]{LaumonCDV}, $\wF\subset F'$. Thus, $\wF=F$ as we wanted to
show.

Now assume $F'/F$ is inseparable. There is a minimal $m\geq 1$ such
that $\delta'=\delta^{p^m}$ is separable over $F$. Since $F'\subset
\Delta$, the power $p^m$ is bounded in terms of $d$. It is clear
that if $\delta$ satisfies the conditions in (\ref{conds}), then
$\delta'$ satisfies similar conditions with $n$ replaced by
$n':=n\cdot p^m$. It is also clear that if $P\in
\mvq(\F_o^{(n)})_{(\wF,\wP)}$ then $P\in
\mvq(\F_o^{(n')})_{(\wF,\wP)}$ and $\delta'$ is its corresponding
element of $\Delta^\times$. This reduced the situation to the
initial case and finishes the proof.
\end{proof}

\begin{prop}\label{prop2.4} The images of supersingular points are
rational over $\F_o^{(d)}:$
$$
\mvq(\F_o^{(d)})_{(F,\Pi)}=\mvq(k)_{(F,\Pi)}.
$$
\end{prop}
\begin{proof} Since $\wF=F$, $h=d$ and there is
a unique $\wo$ dividing $o$, namely $o$ itself. Hence a point $P$ on
$M_{I,o}(k)_{(F,\Pi)}$ is given by $
\Delta^\times[g^{\infty,o}\bKI^{\infty,o},m_{o}]$, and $P$ is
rational over $\F_o^{(d)}$ if there is $\delta\in \Delta^\times$
such that
$$
\left\{
               \begin{array}{ll}
                 (g^{\infty,o})^{-1}\delta g^{\infty,o}\in \bKI^{\infty,o}\\
                 v_o(\Nr(\delta))=d.
               \end{array}
             \right.
$$
Let $f\in F^\times$ be an element whose divisor on $C$ is
$(f)=[o-\deg(o)\infty]$. Take $\delta=f\in F^\times\hookrightarrow
\Delta^\times$, as an element of the center of $\Delta^\times$. Then
$\Nr(f)=f^d$, so $v_o(\Nr(f))=d$. We clearly have $f\in
\bKI^{\infty,o}$, so $(g^{\infty,o})^{-1}f g^{\infty,o}=f\in
\bKI^{\infty,o}$, and the first condition is also satisfied. (Note
that $f$ usually will not be in $\KI^{\infty,o}$.)
\end{proof}

\begin{thm}\label{thm2.5}
Let $S=R\cup \{\infty,o\}$. Let $I=\Spec(A/\fp)\neq o$ be
admissible. For all but finitely many $I$,
$$
\# \mvq(\F_o^{(d)})=\#\PGL_d(\F_\fp)\cdot
\prod_{i=1}^{d-1}\zeta_{F}^S(-i).
$$
\end{thm}
\begin{proof} Let $\bD$ be the $d^2$-dimensional
central division algebra over $F$ whose invariants are
$\inv_o(\bD)=-1/d$, $\inv_\infty(\bD)=1/d$, and
$\inv_x(\bD)=\inv_x(D)$ if $x\neq o, \infty$. In particular, $\bD_x$
is a division algebra for every $x\in S$; cf. \cite[$\S$32]{Reiner}.
Let $\bcD$ be a locally free sheaf of $\cO_C$-algebras with generic
fibre $\bD$, and such that $\bcD_x$ is a maximal order in $\bD_x$
for any $x\in |C|$. Let $\bG$ be the group of units in $\bD$. Let
$\sKI:=\ker(\bcD^\times \to \bcD_I^\times)$, and let
$\sKI^{\infty,o}$ and $\sKI^\infty$ be defined similarly to
$\S$\ref{ss2.2}.

By Theorem \ref{thm-combin},
\begin{equation}\label{eq4.1}
\#\mv(k)_{(F,\Pi)}=\#(\bG(F)\bs \bG(\A^{\infty, o})/\sKI^{\infty,
o}\times \Z).
\end{equation}
(Note that $G(\A^{\infty,o})\cong \bG(\A^{\infty,o})$ and
$\KI^{\infty,o}\cong \sKI^{\infty,o}$). Since $\bD_o$ is a division
algebra, $v_o\circ \Nr:G(F_o)/\bcD_o^\times\cong \Z$; see
\cite[$\S$12]{Reiner}. But $\bG(F)$ acts on $\Z$ in (\ref{eq4.1})
via $v_o\circ \Nr$, so
$$
\bG(F)\bs \bG(\A^{\infty, o})/\sKI^{\infty, o}\times \Z\cong
\bG(F)\bs \bG(\A^{\infty})/\sKI^{\infty}.
$$
Since $\bD_\infty$ ia a division algebra and $\infty$ is rational,
for each $a\in \bG(\A^\infty)$, up to an element of
$\bcD_\infty^\times$, there is a unique $b\in \bG(F_\infty)$ such
that $(a,b)\in \bG^1(\A)$, cf. \cite[$\S$13]{Reiner}. Hence
$$
\bG(F)\bs \bG(\A^{\infty})/\sKI^{\infty}\cong \bG(F)\bs
\bG^1(\A)/\sKI.
$$

If $I\neq \emptyset$ then $\sKI\cap \bG(F)=1$, since only the
constants are everywhere integral but the only constant which
reduces to the identity modulo $I$ is $1$.  Also note that
$\bcD^\times/\sKI\cong \GL_d(\cO_I)$, so with respect to the
canonical product measure on $\bG(\A)$ we have
$$
\#(\bG(F)\bs \bG^1(\A)/\sKI)=\#\GL_d(\cO_I)\cdot \Vol\left(\bG(F)\bs
\bG^1(\A), d\bar{g}\right).
$$
Finally, using Proposition \ref{prop4.1},
$$
\#\mv(k)_{(F,\Pi)}=\frac{\#
\GL_d(\cO_I)}{(q-1)}\prod_{i=1}^{d-1}\zeta_{F}^S(-i).
$$
Since $\F_\fp^\times/\F_q^\times$ acts freely and transitively on
the components of $\mv$ and preserves $\mv(k)_{(F,\Pi)}$,
$$
\#\mvq(k)_{(F,\Pi)}=\#\PGL_d(\F_\fp)\cdot\prod_{i=1}^{d-1}\zeta_{F}^S(-i).
$$
Now the theorem follows from Propositions \ref{prop2.3} and
\ref{prop2.4}.
\end{proof}


\section{Asymptotic Betti numbers}\label{sec4}

Fix a prime $\ell$ not equal to the characteristic $p$ of $F$ and
consider the $\ell$-adic cohomology groups
$$
H^i_{I,o}:=H^i(\mv\otimes_{\F_o}k, \overline{\Q}_\ell), \quad i\geq
0.
$$
Each $H^i_{I,o}$ is a finite dimensional $\overline{\Q}_\ell$-vector
space, which is $0$ for $i\geq 2(d-1)$. Denote $
h^i_{I,o}=\dim_{\overline{\Q}_\ell}H^i_{I,o}$ and
$h_{I,o}=\sum_{i=0}^{2d-2}h^i_{I,o}$. In this section we estimate
how $h_{I,o}$ grows as $\deg(I)$ tends to infinity.

Let $\eta=\Spec(F)$ be the generic point of $C$. We also have the
$\ell$-adic cohomology groups of the generic fibre $M^\cD_{I,\eta}$
of $M^\cD_I$:
$$
H^i_{I,\eta}:=H^i(M^\cD_{I,\eta}\otimes_{F}\bar{F},
\overline{\Q}_\ell), \quad i\geq 0.
$$

By the proper base change theorem, for all $i\geq 0$ there is a
canonical isomorphism of $\overline{\Q}_\ell$-vector spaces
$H^i_{I,\eta}\cong H^i_{I,o}$. Hence we concentrate on estimating
$h_I=\sum_{i=0}^{2d-2} h^i_{I,\eta}$, where
$h^i_{I,\eta}=\dim_{\overline{\Q}_\ell}H^i_{I,\eta}$.

For two $\Q$-valued functions $f(I)$ and $g(I)$ depending on $I$, we
write $f(I)\sim g(I)$ if $f(I)/g(I)\to 1$ as $\deg(I)\to \infty$.

\subsection{Cohomology and automorphic representations}
Let $\cA$ be the space of locally constant
$\overline{\Q}_\ell$-valued functions on the double coset space
$G(F)\bs G(\A)/\varpi_\infty^\Z$. This space is equipped with the
right regular representation of $G(\A)$. Since $D$ is a division
algebra, the coset space $G(F)\bs G(\A)/\varpi_\infty^\Z$ is
compact, so $\cA$ decomposes as a direct sum of irreducible
admissible representations $\Pi$ of $G(\A)$ with finite
multiplicities $m(\Pi)\geq 0$, cf. \cite[p. 291]{LRS}:
\begin{equation}\label{eq-dec}
\cA=\bigoplus_{\Pi} m(\Pi)\cdot\Pi.
\end{equation}
We will refer to the representations appearing in this sum with
non-zero multiplicities as the \textit{automorphic representations}
of $G(\A)$. Each automorphic representation $\Pi$ decomposes as a
restricted tensor product $\bigotimes_{x\in |C|} \Pi_x$ of
irreducible admissible representations of $G(F_x)$. Denote
$\Pi^\infty:=\bigotimes_{x\neq \infty}\Pi_x$, so
$\Pi=\Pi^\infty\otimes \Pi_\infty$.

Among the automorphic representations of $G(\A)$ we have the
\textit{characters} $\chi\circ \Nr$, where $\Nr:G(\A)\to \A^\times$
is the reduced norm and $\chi$ is a Hecke character on $\A^\times$
with $\chi_\infty=1$. These representations are clearly
$1$-dimensional and it is known that they appear with multiplicity
$1$ in the decomposition (\ref{eq-dec}). Any other automorphic
representation is infinite dimensional.

The \textit{Steinberg representation} $\St_\infty$ of
$G(F_\infty)\cong\GL_d(F_\infty)$ is the unique irreducible quotient
of the induced representation
$$
\mathrm{Ind}^{\GL_d(F_\infty)}_{B(F_\infty)}\left(|\cdot|_\infty^{-\frac{d-1}{2}},|\cdot|_\infty^{-\frac{d-3}{2}},
\dots,|\cdot|_\infty^{\frac{d-3}{2}},|\cdot|_\infty^{\frac{d-1}{2}}\right),
$$
where $B\subset \GL_d$ is the standard Borel subgroup of the upper
triangular matrices.

Let $\T_I$ be the Hecke algebra of $\Q$-valued locally constant
functions with compact supports on $G(\A^\infty)$ invariant under
the left and right translation by $\KI^\infty$. (The product on
$\T_I$ is given by the convolution with respect to the Haar measure
on $G(\A^\infty)$ which gives the volume $1$ to the open compact
subgroup $\KI^\infty\subset G(\A^\infty)$.) There is a natural
action of $\T_I$ on each $H^i_{I,\eta}$ which commutes with the
action of $\Gal(\bar{F}/F)$, cf. \cite[$\S$12]{LRS}. Denote by
$(H^i_{I,\eta})^{\mathrm{ss}}$ the associated semi-simplification of
$H^i_{I,\eta}$ as a $\Gal(\bar{F}/F)\times \T_I$-module. One of the
main results in \cite{LRS} relates $(H^i_{I,\eta})^{\mathrm{ss}}$ to
automorphic representations. Taking the $\KI^{\infty}$-invariants in
(14.9), (14.10) and (14.12) of \textit{loc.cit.}, one obtains:

\begin{thm}\label{thm6.1}
$$
(H^i_{I,\eta})^{\mathrm{ss}}=\bigoplus_\Pi V_\Pi^i\otimes
(\Pi^\infty)^{\KI^\infty},
$$ where the sum
is over automorphic representations of $G(\A)$ such that either
$\Pi_\infty=\mathbf{1}$ is the trivial character or
$\Pi_\infty=\St_\infty$, and $V_\Pi^i$ is a finite dimensional
$\overline{\Q}_\ell$ representation of $\Gal(\bar{F}/F)$.

If $\Pi_\infty=\mathbf{1}$ then $\Pi=\chi\circ \Nr$ for a Hecke
character $\chi$ on $\A^\times$. In this case, $V_\Pi^0$ is the
Galois character corresponding to $\chi$ by Class Field Theory,
$V_\Pi^{2i}$ is isomorphic to the Tate twist $V^0_\Pi(-i)$ and
$V_\Pi^{2i+1}=0$, $0\leq i \leq d-1$.

If $\Pi_\infty=\St_\infty$ then $V^i_\Pi=0$ for $i\neq d-1$ and
$\dim_{\overline{\Q}_\ell}V_\Pi^{d-1}=m(\Pi)\cdot d$.
\end{thm}

\begin{cor}\label{cor-new}
$h^0_{I,\eta}=\#\left[F^\times\bs
(\A^\infty)^\times/K^\infty_{\cO_C,I}\right]$.
\end{cor}
\begin{proof} By the strong approximation theorem \cite{Prasad}, the reduced norm
induces an isomorphism
$$
G(F)\bs G(\A^\infty)/\KI^\infty \cong F^\times\bs
(\A^\infty)^\times/K_{\cO_C,I}^\infty.
$$
Combined with Theorem \ref{thm6.1}, this easily implies the claim.
\end{proof}

Let
$$
W_\St(G,I):=\sum_{\Pi}
\dim_{\overline{\Q}_\ell}(\Pi^\infty)^{\KI^\infty},
$$
where the sum is over the automorphic representations of $G(\A)$
with $\Pi_\infty=\St_\infty$. Define $W_\mathbf{1}(G,I)$ similarly.
Since $W_\mathbf{1}(G,I)$ is the number of Hecke characters of
$(\A^\infty)^\times$ of conductor dividing $I$, it is more or less
clear (and will be confirmed by our later calculations) that
$W_\mathbf{1}(G,I) + W_\St(G,I)\sim W_\St(G,I)$.

It is conjectured that the multiplicities $m(\Pi)$ of automorphic
representations are always equal to $1$. If we assume this, then
from the preceding discussion we get the following asymptotic
estimates:
\begin{equation}\label{hToW}
h_I\sim h_{I,\eta}^{d-1} \sim d\cdot W_\St(G,I).
\end{equation}
As is explained in \cite{LRS}, cf. pp. 219 and 310 in \textit{loc.
cit.}, that $m(\Pi)=1$ would follow from the global
Jacquet-Langlands correspondence between $G(F)$ and $\GL_d(F)$. This
correspondence is proven in the literature for global fields of
characteristic zero, cf. \cite[Ch.VI]{HT}. It is very likely that
the theorem is valid in exactly the same formulation in positive
characteristic, but the complete proof is still lacking except for
$d=2$ or $3$. To avoid relying on this yet unproven analogue, we
will deduce (\ref{hToW}) from some results about the cohomology of
quotients of $p$-adic symmetric spaces.

\subsection{Cohomology of varieties having rigid-analytic
uniformization}\label{ss6.2} Let $\Omega^d$ be Drinfeld's
$(d-1)$-dimensional symmetric space over $\Fi$. It is obtained from
the projective $(d-1)$-dimensional space over $\Fi$ by removing all
rational hyperplanes. In \cite{Drinfeld}, Drinfeld showed that this
space has a natural rigid-analytic structure. Let $\G$ be a
discrete, cocompact, torsion-free subgroup of $\PGL_d(\Fi)$. $\G$
naturally acts on $\Omega^d$ and the quotient $X_\G:=\G\bs \Omega^d$
is a proper smooth rigid-analytic variety over $\Fi$, which in fact
is the analytification of a projective algebraic variety $\cX_\G$
over $\Fi$; see \cite{Mustafin}.

Let $\C_\infty$ be the completion of an algebraic closure of $\Fi$.
Assuming there exists an \'etale cohomology theory on the category
of smooth rigid-analytic spaces satisfying four natural properties
\cite[pp. 55-56]{SS}, Schneider and Stuhler computed the groups
$$
H^i(X_\G):=H^i_\et(X_\G\widehat{\otimes}_{\Fi} \C_\infty,
\overline{\Q}_\ell), \quad i\geq 0.
$$
Their result implies the following (see Theorem 4 on page 93 of
\cite{SS}):

\begin{thm}\label{thm-SS} Let $\mu(\G)$ be the multiplicity of the
Steinberg representation $\St_\infty$ in $L^2(\G\bs \PGL_d(\Fi))$.
For $i\neq d-1$
$$
\dim_{\overline{\Q}_\ell}H^i(X_\G)=\left\{
  \begin{array}{ll}
    0, & \hbox{if $i$ is odd or $i>2(d-1)$;} \\
    1, & \hbox{if $i$ is even;}
  \end{array}
\right.
$$
and
$$
\dim_{\overline{\Q}_\ell}H^{d-1}(X_\G)=\left\{
  \begin{array}{ll}
    d\cdot \mu(\G)+1, & \hbox{if $d$ is odd;} \\
    d\cdot
\mu(\G), & \hbox{if $d$ is even.}
  \end{array}
\right.
$$
\end{thm}

In \cite{Berkovich}, Berkovich developed an \'etale cohomology
theory for non-archimedean analytic spaces which satisfies the four
properties required in \cite{SS}. Moreover, he proved a comparison
theorem between analytic and algebraic \'etale cohomology groups
$H^i(X_\G)\cong H^i(\cX_\G\otimes_{\Fi}\bar{F}_\infty,
\overline{\Q}_\ell)$; see \cite[$\S$7.1]{Berkovich}.

\vspace{0.1in}

Now we want to apply Theorem \ref{thm-SS} to $M^\cD_{I,\eta}$. For
this we need to know that the modular varieties $M^\cD_{I,\eta}$
have rigid-analytic uniformization over $\Fi$. As we already
mentioned, the reduced norm induces a bijection
$$
G(F)\bs G(\A^\infty)/\KI^\infty \overset{\sim}{\To}
F^\times\bs(\A^\infty)^\times/\Nr(\KI^\infty)\cong
\cO_I^\times/\F_q^\times.
$$
Choose a system $S$ of representatives for this finite coset space.
For each $s\in S$, let
$$
\G_{I,s}:=G(F)\cap s\KI^\infty s^{-1}.
$$
\begin{lem}\label{lem-gar}
Under the natural embedding $$\G_{I,s}\hookrightarrow
G(F)/F^\times\hookrightarrow G(\Fi)/\Fi^\times\cong \PGL_d(\Fi),$$
$\G_{I,s}$ is a discrete, cocompact, torsion-free subgroup of
$\PGL_d(\Fi)$.
\end{lem}
\begin{proof} It is enough to prove the claim for $\G_I:=\G_{I,1}$.
The fact that $\G_I$ is a discrete, cocompact subgroup of
$\PGL_d(\Fi)$ is well-known, cf. \cite[Prop. 5.3.8]{LaumonCDV}. We
show that it is torsion-free. Assume $\gamma\in \G_I$ is torsion.
Let $n\in \Z_{>0}$ be its exponent, so $\gamma^n=1$. Since $\gamma$
is a nonzero element of the division algebra $D(F)$, we can consider
the finite degree field extension $K=F[\gamma]$ of $F$. We see that
$K$ is obtained by an extension of constants, so $n$ is coprime to
$p$. On the other hand, $\gamma\in \KI^\infty$. Choose a place $x$
which is in the support of $I$. Since $x\not \in R$ by assumption,
we can consider $\gamma$ as an element of the principal congruence
subgroup $\G(x):=\{M\in \GL_d(\cO_x)\ |\ M\equiv 1\ \mod\
\varpi_x\}$. It is not hard to check that $\G(x)$ has no
prime-to-$p$ torsion. Therefore, $n=1$.
\end{proof}

Next, since
\begin{equation}\label{eq-sat}
G(F)\bs G(\A)/\KI^\infty \Fi^\times \cong \bigsqcup_{s\in S}
\left(\G_{I,s}\bs \PGL_d(\Fi)\right),
\end{equation}
we have
\begin{equation}\label{eq-Stmu}
W_\St(G,I)=\sum_{s\in S} \mu(\G_{I,s}).
\end{equation}

\begin{thm}\label{thm-BS}
$$
(M^\cD_{I,\eta}\otimes_F \Fi)^\an\cong \bigsqcup_{s\in S}
\left(\G_{I,s}\bs \Omega^d \right).
$$
\end{thm}
\begin{proof} This follows by applying Raynaud's ``generic fibre''
functor to Theorem 4.4.11 in \cite{BS}, which is stated in the
language of formal schemes. The proof of this theorem is only
outlined in \cite{BS}. Nevertheless, as is shown in \cite{Spiess},
the desired uniformization can be deduced from Hausberger's version
of the Cherednik-Drinfeld theorem for $M^\cD_{I}$ \cite{Hausberger}.
\end{proof}

Since $F\hookrightarrow \Fi$, by the proper base change theorem we
have an isomorphism $H^i_{I,\eta}\cong
H^i(M^\cD_{I,\eta}\otimes_{F}\bar{F}_\infty, \overline{\Q}_\ell)$
for all $i\geq 0$. Combining this with Theorem \ref{thm-SS},
(\ref{eq-Stmu}), Theorem \ref{thm-BS} and Berkovich's results, we
easily obtain (\ref{hToW}). Note also that by comparing the
dimensions of cohomology groups in Theorems \ref{thm6.1} and
\ref{thm-SS}, we can deduce that the multiplicities $m(\Pi)$ are
indeed $1$.

\subsection{Euler-Poincar\'e measure} It remains to estimate
$W_\St(G,I)$. We start with a well-known fact about discrete
subgroups of $p$-adic groups.

\begin{thm}
Let $\G$ be a discrete, cocompact, torsion-free subgroup of
$\PGL_d(\Fi)$. Then $\dim_{\Q} H^0(\G,\Q)=1$, $\dim_{\Q}
H^{d-1}(\G,\Q)=\mu(\G)$, and $H^i(\G,\Q)=0$ for $i\neq 0,\ d-1$. In
particular,
\begin{equation}\label{Garland}
\chi(\G):=\sum_{i\geq 0} (-1)^i \dim_{\Q}
H^i(\G,\Q)=1+(-1)^{d-1}\mu(\G).
\end{equation}
\end{thm}
\begin{proof}
This was first proven by Garland \cite{Garland} using a discrete
analogue of curvature on Bruhat-Tits buildings, but under the
assumption that $q$ is large enough. Casselman \cite{Casselman}
removed the restriction on $q$ by giving a completely different
proof which relies on representation-theoretic methods.
\end{proof}

From (\ref{Garland}) and (\ref{eq-Stmu}), we get
\begin{equation}\label{eq-3.4}
W_\St(G,I)\sim \sum_{s\in S}(-1)^{d-1}\chi(\G_{I,s}).
\end{equation}

In \cite{SerreCGD}, Serre developed a theory which allows to compute
$\chi(\G_{I,s})$ as a volume. Serre's result is reproduced in a
convenient form in Proposition 5.3.6 of \cite{LaumonCDV}. Combining
this with Proposition 5.3.9 in \cite{LaumonCDV}, in our situation we
get the following statement:

Let $dg$ be the Haar measure on $\GL_d(\Fi)$ normalized by
$\Vol(\GL_d(\cO_\infty), dg)=1$. Let $dz$ be the Haar measure on
$\Fi^\times$ normalized by $\Vol(\cO_\infty^\times, dz)=1$. Fix the
Haar measure $dh=dg/dz$ on $\PGL_d(\Fi)$ and the counting measure
$d\delta$ of $\G_{I,s}$. Then
\begin{equation}\label{eq-3.5}
\chi(\G_{I,s})=\Vol\left(\G_{I,s}\bs \PGL_d(\Fi),
\frac{dh}{d\delta}\right)\cdot
\frac{1}{d}\prod_{i=1}^{d-1}\zeta_\infty(-i)^{-1}.
\end{equation}

The push-forward of the canonical product measure $d\bar{g}$ on
$G(\A)$ to the double coset space $G(F)\bs G(\A)/\KI^\infty
\Fi^\times$ induces via (\ref{eq-sat}) the measure $dh/d\delta$ on
each $\G_{I,s}\bs \PGL_d(\Fi)$. Combining this with (\ref{eq-3.4})
and (\ref{eq-3.5}), we get
\begin{equation}\label{eq-Wasm}
W_\St(G,I)\sim (-1)^{d-1}\Vol\left(G(F)\bs G(\A)/\KI^\infty
\Fi^\times, d\bar{g}\right)\cdot
\frac{1}{d}\prod_{i=1}^{d-1}\zeta_\infty(-i)^{-1}.
\end{equation}

It remains to compute this volume. It is well-known that
$\norm{\cdot}:G(\A)\to q^{\Z}$ in (\ref{eq-norm}) is surjective. The
image of $\Fi^\times$ in $q^{\Z}$ is clearly $q^{d\Z}$. Hence there
is an exact sequence
$$
0\to G(F)\bs G^1(\A)/\cO_\infty^\times\to G(F)\bs G(\A)/\Fi^\times
\to \Z/d\Z\to 0,
$$
which implies
\begin{equation}\label{eqVol} \Vol\left(G(F)\bs
G(\A)/\KI^\infty \Fi^\times, d\bar{g}\right) = d\cdot\#
\GL_d(\cO_I)\cdot \Vol\left(G(F)\bs G^1(\A), d\bar{g}\right).
\end{equation}

Combining (\ref{hToW}), (\ref{eq-Wasm}), (\ref{eqVol}) and
Proposition \ref{prop4.1}, we obtain the following asymptotic
formula
\begin{equation}\label{eq3.6}
h_I\sim (-1)^{(d-1)}d\cdot \frac{\#
\GL_d(\cO_I)}{(q-1)}\prod_{i=1}^{d-1}\zeta^{R\cup\infty}_F(-i).
\end{equation}

\vspace{0.1in}

Now assume $I=\Spec(A/\fp)\neq o$ is admissible. Let $\bar{h}_I$ be
the sum of the Betti numbers of $\mvq$. Since the morphism $\mv\to
\mvq$ is \'etale of degree $[\F_\fp:\F_q]$, (\ref{eq3.6}) implies
\begin{equation}\label{eq-asm1}
\bar{h}_I\sim (-1)^{(d-1)}d\cdot
\#\PGL_d(\F_\fp)\prod_{i=1}^{d-1}\zeta^{R\cup\infty}_F(-i).
\end{equation}
On the other hand, by Theorem \ref{thm2.5} we know that
\begin{equation}\label{eq-asm2}
\# \mvq(\F_o^{(d)})\sim \#\PGL_d(\F_\fp)\cdot
\prod_{i=1}^{d-1}\zeta_{F}^{R\cup \infty\cup o}(-i).
\end{equation}
Combining (\ref{eq-asm1}) and (\ref{eq-asm2}), we get
$$
\# \mvq(\F_o^{(d)})/\bar{h}_I\sim
\frac{1}{d}(-1)^{d-1}\prod_{i=1}^{d-1}\zeta_o(-i)^{-1}=\frac{1}{d}\prod_{i=1}^{d-1}
(q_o^i-1),
$$
which is the second claim of Theorem \ref{thm-main}. The first claim
is Corollary \ref{cor-SF}.



\end{document}